\documentclass{birkjour}
\usepackage{amsfonts,amssymb,amsmath,amsthm}
\usepackage{url}
\usepackage{enumerate}

\newtheorem{thrm}{Theorem}[section]
\newtheorem{lem}[thrm]{Lemma}
\newtheorem{prop}[thrm]{Proposition}

\theoremstyle{definition}
\newtheorem{definition}[thrm]{Definition}
\newtheorem{remark}[thrm]{Remark}
\newtheorem{example}[thrm]{Example}
\numberwithin{equation}{section}

\def\sqr#1#2{{\,\vcenter{\vbox{\hrule height.#2pt\hbox{\vrule width.#2pt
height#1pt \kern#1pt\vrule width.#2pt}\hrule height.#2pt}}\,}}

\author[C-H. Chu]{Cho-Ho Chu}
\address{
School of Mathematical Sciences,
Queen Mary, University of London,
London E1 4NS, UK}
\email{c.chu@qmul.ac.uk}
\author[L. Li]{Lei Li}
\address{School of Mathematical Sciences and LPMC,
  Nankai University, Tianjin 300071,
 China}
\email{leilee@nankai.edu.cn}

\thanks{The second author is supported by NSF(China) research grants No.11301285 and No.11371201}

\keywords{Separably injective Banach space. C*-algebra. Substonean space. F-space.
} \subjclass{Primary 46B20, 54G05, 17C65}

\begin{document}

\title[Separably injective C*-algebras]
 {Separably injective C*-algebras}
\begin{abstract}
We show  that a C*-algebra is a $1$-separably
injective Banach space if, and only if, it is linearly isometric to the Banach
space $C_0(\Omega)$ of complex continuous functions vanishing at
infinity on a substonean locally compact Hausdorff space $\Omega$.
\end{abstract}

\maketitle

\section{Introduction}

\label{sect1}
It is well-known that there are few examples of $1$-injective Banach
spaces. These are Banach spaces $V$ for which every continuous
linear map  $T: Y \rightarrow V$ on a Banach space $Y$ admits a
norm preserving extension to a super space $Z \supset Y$,
equivalently, contractive linear maps $T: Y \rightarrow V$ extend
to contractive ones on $Z$. Indeed, a Banach space $V$ is
$1$-injective if and only if it is linearly isometric to the
continuous function space $C(\Omega)$ on some stonean space
$\Omega$ \cite{H58,kelley}. If the $1$-injectivity condition is
relaxed to requiring that each continuous linear map $T: Y
\rightarrow V$ extends to a continuous one on $Z\supset Y$, then
it is unclear what constitutes the larger class of Banach spaces, called
{\it $\lambda$-injective} or {\it $\mathcal{P}_\lambda$ spaces},
satisfying this condition. However, one can consider the class of
{\it $\lambda$-separably injective} Banach spaces to which only
continuous linear maps on {\it separable} spaces are extendable to
{\it separable} super spaces. Of particular interest is the
subclass of {\it $1$-separably injective}
Banach spaces to which
{\it contractive} linear maps on separable spaces admit {\it contractive}
extension on separable super spaces. While $c_0$ is the only
$\lambda$-separably injective space among infinite dimensional
{\it separable} Banach spaces \cite{z}, it has been shown recently
in \cite{AC13} that among {\it nonseparable} real Banach spaces, there
are indeed many interesting examples of $\lambda$-separably
injective spaces. In particular, the Banach space $C(\Omega, \mathbb{R})$ of real continuous
functions on a compact Hausdorff space $\Omega$ is $1$-separably injective if, and only if,
$\Omega$ is an F-space. It is natural to ask if this result also holds for the space
$C_0(S)$ of continuous functions vanishing at infinity on a locally compact space $S$. This case has not been
discussed in \cite{AC13} and in fact, the example of $c_0$ provides a negative answer since
$\mathbb{N}$ is an F-space, but $c_0$ is not $1$-separably injective although it is $2$-spearably injective.

In this paper, we give a complete answer to the above question and prove, more generally, that
a C*-algebra is $1$-separably injective if, and only if, it is linearly isometric to the Banach
space $C_0(S)$ of complex continuous functions vanishing at infinity on a {\it substonean} locally compact
Hausdorff space $S$. Particularly, abelian
monotone sequentially complete C*-algebras are $1$-separably injective.
This example may be of interest as the class of monotone complete C*-algebras is closely related to generic dynamics
\cite{sw}.

\section{Separably injective Banach spaces}

The concept
of a separably injective {\it real} Banach space, considered in
\cite{AC13}, can be extended naturally to that for a complex
Banach space.

\begin{definition}
A  complex (resp.\,real) Banach space $V$ is said to be
\textit{$1$-separably injective} if for every complex (resp.\,real) {\it
separable} Banach space $Z$ and every closed subspace $Y\subset
Z$, every bounded linear operator $T: Y\to V$ extends to a bounded
linear operator $\widetilde{T}: Z\to V$ with $\|\widetilde{T}\|=
\|T\|$.
\end{definition}

Given a locally compact Hausdoff space $\Omega$, we will denote by
$C_0(\Omega)$ the abelian C*-algebra of complex continuous functions on
$\Omega$ vanishing at infinity. If $\Omega$ is compact, then we
omit the subscript $0$ and denote by $C(\Omega,\mathbb{R})$ the
Banach space of real continuous functions on $\Omega$.

\begin{definition}
Let $\Omega$ be a  locally compact Hausdorff space. It is called an {\it F-space}
if for each real continuous function $f$ on $\Omega$, there is a real continuous
function $k$ on $\Omega$ such that $f = k|f|$ (cf.\,\cite[14.25]{GJ}). Following
\cite{GP}, we call $\Omega$   \textit{substonean} if any two
disjoint open $\sigma$-compact subsets of $\Omega$ have disjoint
compact closures.
\end{definition}

The compact substonean spaces are exactly the compact \textit{F-spaces}.
 However, infinite discrete spaces are
F-spaces without being substonean. We refer to \cite[Example
5]{HW89} for an example of a substonean space which is not an
F-space.

\begin{example}\label{s} Let $\Omega$ be a compact Hausdorff space.
Using the results in \cite{h,L}, it has been shown in
\cite[Proposition 4.2]{AC13} that the real continuous function
space $C(\Omega,\mathbb{R})$ is $1$-separably injective if, and only if,
$\Omega$ is an F-space. This result remains true if we replace
$C(\Omega,\mathbb{R})$ by the complex continuous function space
$C(\Omega)$. Indeed, if $C(\Omega,\mathbb{R})$ is $1$-separably
injective and given a contractive complex linear operator $T: Y
\rightarrow C(\Omega)$, where $Y$ is a closed subspace of a
separable complex Banach space $Z$, the real part $ {\rm Re}\, T :
 y\in Y_r \mapsto {\rm Re}\, T(y)\in C(\Omega,\mathbb{R})$ extends to a real linear
contraction $T_r: Z_r \rightarrow C(\Omega, \mathbb{R})$ which, as
in the proof of \cite[Theorem 2]{H58}, gives a complex linear
contraction
$$z\in Z \mapsto T_r(z) - i T_r(iz) \in C(\Omega)$$
extending $T$ since $T(y) = {\rm Re}\, T(y)-i {\rm Re}\, T(i y)$
for $y\in Y$. Hence $C(\Omega)$ is $1$-separably injective. Conversely, if $C(\Omega)$
is $1$-separably injective, it will follow from
Theorem \ref{111} that $C(\Omega, \mathbb{R})$ is $1$-separably injective.

However, as noted earlier, the above result is not valid for the space $C_0(S)$ of
continuous functions vanishing at infinity on a locally compact
Hausdorff space $S$.
Separable injectivity of $C_0(S)$ has not been considered in \cite{AC13}.
A topological
criterion for $1$-separable injectivity of $C_0(S)$ follows from  Theorem \ref{111}.
\end{example}

Let $V$ be a complex $1$-separably injective Banach space and let $T:Y
\rightarrow V$ be a bounded linear operator on a closed subspace
$Y$ of a complex Banach space $X$.  The arguments in the proof of
\cite[Proposition 3.5 (a)]{AC13} for {\it real} $1$-separably
injective spaces can be extended to the complex case and one can
show  that $T$ has a norm preserving extension $\widetilde T : X
\rightarrow V^{**}$. A further application of the arguments in the
proof of \cite[Theorem 2.1, (9)$\Rightarrow$(1)]{L64}, which are
also valid for complex spaces, gives the following result. The result for real
$1$-separably injective spaces has been shown in \cite{AC13}.

\begin{lem}\label{lem}
Let $V$ be a $1$-separably injective complex Banach space. Then the bidual $V^{**}$
is $1$-injective.
\end{lem}

\section{Separably injective C*-algebras}

We characterize $1$-separably injective C*-algebras in this section. We begin with a simple
lemma.

\begin{lem} \label{le}
A $1$-separably injective C*-algebra is abelian.
\end{lem}

\begin{proof}
Let $A$ be a $1$-separably injective C*-algebra.
 By Lemma \ref{lem},
the bidual $A^{**}$ is $1$-injective and hence linearly isometric
to a continuous function space $C(\Omega)$ on some stonean
space $\Omega$ \cite{H58}. The linear isometry between the C*-algebras
$A^{**}$ and $C(\Omega)$ preserves the Jordan triple product
$$\{a,b,c\} := \frac{1}{2}(ab^*c + cb^*a) \qquad (a,b,c \in A^{**})$$
by a well-known result of Kadison \cite{k} (see also \cite[Theorem 3.1.7]{book}).
Since $C(\Omega)$ is an abelian algebra, we must have, via the isometry between
$C(\Omega)$ and $A^{**}$,
$$\{a,b,\{c,\mathbf 1, \mathbf 1\}\} = \{a, \{b,c,\mathbf 1\}, \mathbf 1\}$$
for $a,b,c \in A^{**}$, where $\mathbf 1$ denotes the identity in $A^{**}$.
Let $p$ be a projection in $A^{**}$ and $a\in A^{**}$. A simple computation gives
\begin{eqnarray*}
\frac{1}{2} (pa + ap) &=& \{p,p,\{a,\mathbf 1, \mathbf 1\}\} =\{p,\{p,a,\mathbf 1\}, \mathbf 1\} \\
&=& \frac{1}{4}( pa + ap + 2pap)
\end{eqnarray*}
and $pa +ap = 2pap$, which implies $pa=ap$. Hence $A^{**}$ is abelian since it is generated by
projections. In particular,  $A$ itself is abelian.
\end{proof}

We have the following result readily.

\begin{prop}
Let $A$ be a von Neumann algebra. The following conditions are equivalent.
\begin{itemize}
\item[(i)] $A$ is $1$-separably injective. \item[(ii)] $A$ is
$1$-injective.
\end{itemize}
\end{prop}
\begin{proof}
(i) $\Rightarrow$ (ii). By the above observation, the unital algebra $A$ is abelian and
hence linearly isometric
to a continuous function space $C(\Omega)$ on some compact
Hausdorff space $\Omega$. Since $A$ has a predual, $\Omega$ must be hyperstonean and therefore $C(\Omega)$
is $1$-injective  by \cite{H58}.

\end{proof}

\begin{remark} The above proposition is false for unital C*-algebras.
Indeed, let $\beta \mathbb{N}$ be the Stone-\v{C}ech
compactification of $\mathbb{N}$. Then $\beta \mathbb{N}\backslash
\mathbb{N}$ is a compact F-space by \cite[p.\,210]{GJ}. Hence the
C*-algebra $C(\beta \mathbb{N}\backslash \mathbb{N})$ is $1$-separably
injective (cf.\,Example \ref{s}), but not $1$-injective since $\beta
\mathbb{N}\backslash \mathbb{N}$ is not stonean \cite[p.\,98]{GJ}.
\end{remark}

We now determine the class of $1$-separably injective
C*-algebras. A useful fact noted in \cite[Proposition 1.1]{GP} is
that a locally compact Hausdorff space $S$ is substonean if, and
only if, the following condition holds: given $f$ and $g$ in
$C_0(S)$ satisfying $fg = 0$, there are functions $f_1, g_1 \in
C_0(S)$ such that $f_1g_1=0$, $f_1 f = f$ and $g_1 g = g$. We will
need the following definition introduced in \cite{HW89}.

\begin{definition}
A nonempty subset $S_0$ of a topological space $S$ is called a
\textit{P-set} if it is closed and any $G_\delta$-set containing
$S_0$ is a neighborhood of $S_0$.  A point $p\in S$ is called a
\textit{P-point} if $\{p\}$ is P-set in $S$.
\end{definition}

It has been remarked in \cite{HW89} that a nonempty subspace $S_0$
of $S$ is a P-set if and only if each real continuous function $f$
on $S$, vanishing on $S_0$, must vanish on a neighborhood of
$S_0$. If $S$ is a locally compact and noncompact space, then $S$
is substonean if, and only if, the one-point compactification $S\cup
\{\infty\}$ is an F-space and $\infty$ is a P-point in $S\cup
\{\infty\}$ (cf.\,\cite[Theorem 1]{HW89}).

\begin{thrm}\label{111}
Let $A$ be a C*-algebra. The following conditions are equivalent.
\begin{itemize}
\item[(i)] $A$ is $1$-separably injective. \item[(ii)] $A$ is linearly
isometric to the Banach space $C_0(S)$ of complex continuous
functions vanishing at infinity on a substonean locally compact
Hausdorff space $S$.
\end{itemize}
\end{thrm}

\begin{proof}
(i) $\Rightarrow$ (ii). By Lemma \ref{le}, $A$ is abelian and hence
linearly isometric to the function space $C_0(S)$ on some locally compact Hausdorff space
$S$.  We show that $S$ is substonean.

Given any function $f\in C_0(S)$, we define the {\it cozero set}
of $f$ to be the set $coz(f) = \{x\in S: f(x) \neq 0\}$. Let $U$
and $V$ be two disjoint open $\sigma$-compact sets in $S$. We show
that they have disjoint compact closures. It has been observed in
\cite[p.125]{GP} that one can find, via Urysohn's lemma, two
functions $f,g\in C_0(S)$ with $0\leq f,g\leq 1$ such that $U =
coz(f)$ and $V = coz(g)$. We note that $fg=0$ in this case.

Let $h=\chi_{coz(f)}$ be the characteristic function of $coz(f)$.
Let $Y$ be the closed linear span of $\{f^{1/n}, g^{1/n}: n = 1,2,
\ldots .\}$ in $C_0(S)$, and let $Z = Y +\mathbb{C}\,h \subset
\ell^\infty(S)$. Since $C_0(S)$ is $1$-separably injective,  the
identity map $\iota: Y\to C_0(S)$ admits a norm preserving
extension $\tilde{\iota}: Z\to C_0(S)$. Write
$k=\tilde{\iota}(h)\in C_0(S)$ and note that $\|k\| \leq 1$.

For each $n\in \mathbb{N}$, we have $\|h-2f^{\frac{1}{n}}\|\leq 1$
and therefore
$\|k-2f^{\frac{1}{n}}\|=\|\tilde{\iota}(h)-2\tilde{\iota}(f^{\frac{1}{n}})\|\leq
1$. In particular, $|k(x)-2f^{\frac{1}{n}}(x)|\leq 1$ for all
$n\in \mathbb{N}$ which, together with $|k(x)| \leq 1$, implies
$k(x)=1$ for $x\in coz(f)$. It follows that  $kf=f$.

Since $hg=0$, we have $\|h+e^{i\theta}g^{\frac{1}{n}}\|\leq 1$ for
every $\theta\in [0, 2\pi)$ which gives
$\|k+e^{i\theta}g^{\frac{1}{n}}\|=\|\tilde{\iota}(h)+
\tilde{\iota}(e^{i\theta}g^{\frac{1}{n}})\|\leq 1$. Hence for each
$x\in coz(g)$, one has $|k(x)+e^{i\theta} g^{\frac{1}{n}}(x)|\leq
1$ for all $n\in\mathbb{N}$ and $\theta\in [0, 2\pi)$. This
implies $k(x)=0$ for $x \in coz (g)$ and therefore $kg=0$.

We have $coz(f)\subset \{x\in S: k(x)=1\}$ and $coz(g)\subset
\{x\in S: k(x)=0\}$, where $\{x\in S: k(x)=1\}$ is compact and
contained in the open set $\{x\in S: |k(x)|>\frac{3}{4}\}$.
Applying Urysohn's lemma to the compact set
$\overline{coz(f)}\subset \{x\in S: k(x)=1\}$, one can find a
function $f_1\in C_0(S)$ such that $0\leq f_1\leq 1$, $f_1=1$ on
$\overline{coz(f)}$ and $f_1=0$ outside $\{x\in S:
|k(x)|>\frac{3}{4}\}$.

Considering the  function $h'=\chi_{coz(g)}$ with similar
arguments, we can find another function $k'\in C_0(S)$ with
$\|k'\| \leq 1$ such that $coz(g)\subset \{x\in S: k'(x)=1\}$ and
hence $\overline{coz(g)}$ is a compact subset of $S$. Since
$\{x\in S: |k(x)|<\frac{1}{4}\}$ is  an open subset of $S$
containing $\overline{coz(g)}$, Urysohn's lemma again yields a
function $g_1\in C_0(S)$ such that $0\leq g_1\leq 1$, $g_1=1$ on
$\overline{coz(g)}$ and $g_1=0$ outside $\{x\in S:
|k(x)|<\frac{1}{4}\}$. It follows that $f_1g_1=0$ and $f_1f=f,
g_1g=g$.

It follows that the closures $\overline{U}=\overline{coz(f)}
\subset \{x\in S: f_1(x) \geq 1/2\}$ and
$\overline{V}=\overline{coz(g)} \subset \{x\in S: g_1(x) \geq
1/2\}$ are compact and disjoint.

(ii) $\Rightarrow$ (i). Let $A$ be linearly isometric to  $C_0(S)$
where $S$ is a substonean locally compact Hausdorff space. We show
that $C_0(S)$ is separably injective. This is true if $S$ is
compact, as shown in Example \ref{s}. Let $S$ be noncompact. Since
the one-point compactification $S\cup\{\infty\}$ is an $F$-space,
$C(S\cup\{\infty\})$ is $1$-separably injective, again by Example
\ref{s}.

Let $Y$ be a closed subspace of a separable Banach space $Z$ and
let $T: Y\to C_0(S)$ be a bounded linear operator. We identify
$C_0(S)$ with the closed subspace
$\{u\in C(S \cup\{\infty\}): u(\infty)=0\}$
of  $C(S \cup\{\infty\})$. Let $\widetilde{T}: Z\to C(S\cup
\{\infty\})$ be a norm preserving extension of $T$.

Let $\{f_n: n \in\mathbb{N}\}$ be a countable dense subset  of
$T(Y)\subset C_0(S)$. Since $S$ is substonean, $\infty$ is a
P-point and there exists an open neighbourhood $U_n$ of $\infty$
such that $f_n =0$ on $U_n$.

Moreover, the $G_\delta$-set $\bigcap_n U_n$ contains an open
neighbourhood $U$ of $\infty$. Hence Urysohn's lemma again enables
us to choose a function $e\in C_0(S)$ with $\|e\| =1$ such that
$e=1$ on $S\backslash U$ and $e(\infty)=0$. This gives
 $f_ne=f_n$ for all $n\in\mathbb{N}$.

It can now be seen readily that the linear map ${T}_e: Z\to
C_0(S)$ defined by
\[{T}_e(z)=\widetilde{T}(z)e \qquad (z\in Z).\]
is a norm preserving extension of $T$.
\end{proof}

We conclude by mentioning some interesting examples of substonean
spaces. For any locally compact, $\sigma$-compact Hausdorff space
$S$ with Stone-\v{C}ech compactification $\beta S$, the space
$\beta S \backslash S$ is a compact F-space \cite[14.27]{GJ}.
Locally compact substonean spaces include the {\it Rickart spaces}
which are exactly those locally compact spaces $S$ for which
$C_0(S)$ is monotone sequentially complete \cite{GP}. In
particular, abelian
monotone sequentially complete C*-algebras are $1$-separably injective.

\end{document}